\UseRawInputEncoding

\documentclass[a4paper,12pt]{article}

\usepackage{amsmath,amsfonts,amssymb,amsthm,amscd}
\usepackage {amsfonts, amssymb, amscd,amsthm, amsmath, eucal}

\usepackage{amsbsy}
\usepackage[all]{xy}

\theoremstyle{plain} {
  \newtheorem{thm}{Theorem}[section]
  
  \newtheorem{cor}[thm]{Corollary}
  
  \newtheorem{prop}[thm]{Proposition}
  \theoremstyle{definition}
  \newtheorem{rem}[thm]{Remark}

  \theoremstyle{plain}

  \newtheorem{sssection}[thm]{}
}

{

}
\renewcommand{\subsubsection}{\sssection\rm}

\newcommand{\Aff}{\mathbf {A}}

\renewcommand \phi\varphi

\begin{document}

\title{On Grothendieck--Serre conjecture in mixed characteristic for $SL_{1,D}$
}

\author{Ivan Panin\footnote{February the 10-th of 2022}
}

\date{Steklov Mathematical Institute at St.-Petersburg}

\maketitle

\begin{abstract}
Let R be an unramified regular local ring of mixed characteristic, D an Azumaya R-algebra,
K the fraction field of R, $Nrd: D^{\times} \to R^{\times}$ the reduced norm homomorphism.
Let $a \in R^{\times}$ be a unit. Suppose the equation $Nrd=a$ has a solution over K, then
it has a solution over R.

Particularly, we prove the following. Let R be as above and $a,b,c$ be units in $R$. Consider the equation
$T^2_1-aT^2_2-bT^2_3+abT^2_4=c$.
If it has a solution over $K$, then it has a solution over $R$.

Similar results are proved for regular local rings, which are
geometrically regular over a discrete valuation ring. These results
extend result proven in \cite{PS} to the mixed characteristic case.
\end{abstract}

\section{Introduction}
A well-known conjecture due to J.-P.~Serre and A.~Grothendieck
\cite[Remarque, p.31]{Se},
\cite[Remarque 3, p.26-27]{Gr1},
and
\cite[Remarque 1.11.a]{Gr2}
asserts that, for any regular local ring $R$ and
any reductive group scheme $G$ over $R$ rationally trivial
$G$-homogeneous spaces are trivial. Our results correspond
to the case  when $R$ an unramified regular local ring of mixed characteristic and
$G$ is  the group $\text{SL}_1(D)$ of norm one elements of an Azumaya
$R$-algebra $D$. Our results extend to the mixed characteristic case
the ones proven in \cite{PS} by A.Suslin and the author.
Note that our Theorem \ref{thm2} is {\it much stronger} than \cite[Theorem 8.1]{P1}.
Also, details of the proof are given better in this preprint.

Our approach is this. Using the D.Popescu theorem we reduce the question to
the case when R is essensially smooth over $\mathbb Z_{(p)}$ and $D$ is defined over R.
Then, using a geometric presentation lemma due to Cesnavicius \cite[Proposition 4.1]{C}
and the method as in \cite{PS} we prove a purity result for the functor
$K_1(-,D)$. Finally, a diagram chasing shows the result mentioned above.

A rather good survey of the topic is given in \cite{Pan4}. Point out
the conjecture is solved in the case, when $R$ contains a field.
More precisely, it is solved when $R$ contains an infinite field, by R.Fedorov and the author in \cite{FP}.
It is solved by the author in the case, when $R$ contains a finite field in \cite{Pan3} (see also \cite{P}).

The case of mixed characteristic is widely open. Here are several references.

\smallskip $\bullet$ The case when the group scheme is $\text{PGL}_n$ and the ring $R$ is an arbitrary regular local ring is done by A.Grothendieck in 1968 in \cite{Gr2}.

\smallskip $\bullet$ The case of an arbitrary reductive group scheme over a discrete valuation ring or over a henselian ring is solved by Y.~Nisnevich in 1984 in~\cite{N}.

\smallskip $\bullet$ The case, where $\bf G$ is an arbitrary torus over a regular local ring, was settled by J.-L.~Colliot-Th\'{e}l\`{e}ne and J.-J.~Sansuc in 1987 in ~\cite{C-T/S}.

\smallskip $\bullet$ The case, when $\bf G$ is quasi-split reductive group scheme over arbitrary two-dimensional local rings,  is solved by Y.~Nisnevich in 1989 in~\cite{Ni2}.

\smallskip $\bullet$ The case when $G$ the unitary group scheme $\text{U}^{\epsilon}_{A,\sigma}$ is solved by S.Gille an the author recently in \cite{GiP};
here $R$ is an unramified
regular local ring of characteristic $(0,p)$ with $p\neq 2$ and $(A,\sigma)$ is an Azumaya $R$-algebra with involution;

\smallskip $\bullet$ In \cite{PSt} the conjecture is solved for any semi-local Dedekind
domain providing that $G$ is simple simply-connected and $G$ contains a torus $\mathbb G_{m,R}$;

\smallskip $\bullet$ The latter result is extended in \cite{NG} to arbitrary reductive group schemes $G$ over any semi-local Dedekind domain;

\smallskip $\bullet$ There are as two very interesting recent publications \cite{Fe1}, \cite{Fe2} by R.Fedorov;

\smallskip $\bullet$ The case, when $\bf G$ is quasi-split reductive group scheme over an unramified regular local ring
is solved recently by K.Cesnavicius in \cite{C};

\section{Agreements}\label{agreements}
Through the paper \\
$A$ is a d.v.r., $m_A\subseteq A$ is its maximal ideal;\\
$\pi \in m_A$ is a generator of the maximal ideal; \\
$k(v)$ is the residue field $A/m_A$;\\
$p>0$ is the characteristic of the field $k(v)$; \\
it is supposed in this preprint that the fraction field of $A$ has characteritic zero;\\
$d\geq 1$ is an integer;\\
$X$ is an irreducible $A$-smooth affine $A$-scheme of relative dimension $d$;\\
{\bf So, all open subschemes of $X$ are regular and all its local rings are regular}.\\
If $x_1,x_2,\dots,x_n$ are closed points in the scheme $X$, then write \\
$\mathcal O$ for the semi-local ring $\mathcal O_{X,\{x_1,x_2,\dots,x_n\}}$, $\mathcal K$ for the fraction field of $\mathcal O$, \\
$U$ for $Spec(\mathcal O)$, $\eta$ for $Spec(\mathcal K)$.\\
Recall that a regular local ring $R$ of mixed characteristic $(0,p)$ is called unramified
if the ring $R/pR$ is regular;\\
recall from [SP, 0382] that a Noetherian algebra over a field $k$ is geometrically regular if its base change to
every finite purely inseparable (equivalently, to every finitely generated) field extension of $k$ is regular; \\
one says that a regular local $A$-algebra $R$ is geometrically regular if the $k(v)$-algebra $R/mR$ is geometrically regular.

\section{Main results}\label{sect: main_result}
\begin{thm}\label{thm1}
Let $R$ be an unramified regular semi-local domain of mixed characteristic $(0,p)$, $D$ an Azumaya $R$-algebra,
$K$ the fraction field of $R$, $Nrd: D^{\times} \to R^{\times}$ the reduced norm homomorphism.
Let $a \in R^{\times}$ be a unit. Suppose the equation $Nrd=a$ has a solution over $K$, then
it has a solution over $R$.
\end{thm}
The following result is an extension of Theorem \ref{thm1}.
\begin{thm}\label{thm2}
Let $R$ be a geomerically regular semi-local integral $A$-algebra. Let $D$ be an Azumaya $R$-algebra,
$K$ the fraction field of $R$, $Nrd: D^{\times} \to R^{\times}$ the reduced norm homomorphism.
Let $a \in R^{\times}$ be a unit. Suppose the equation $Nrd=a$ has a solution over $K$, then
it has a solution over $R$.
\end{thm}

\begin{cor}
Let $R$ be a geomerically regular semi-local integral $A$-algebra, $K$ the fraction field of $R$.
Let $a,b,c$ be units in $R$. Consider the equation
\begin{equation}\label{3_units}
T^2_1-aT^2_2-bT^2_3+abT^2_4=c.
\end{equation}
If it has a solution over $K$, then it has a solution over $R$.
\end{cor}

\begin{rem}
Let $a,b,c$ be units in $R$ as in the Corollary. Let $D$ be the generalised quaternion $R$-algebra
given by generators $u$, $w$ and relations $u^2=a$, $w^2=b$, $uw=-wu$. Then the reduced norm
$Nrd: D\to R$ takes a quaternion $\alpha+\beta u+ \gamma w+ \delta uw$ to the element
$\alpha^2 - a\beta^2 - b\gamma^2 + ab\delta^2$. This is why the Corollary is a consequence
of Theorem \ref{thm2}.
\end{rem}

\section{Cesnavicius geometric presentation proposition}\label{thms:geometric}
Let $A$ be a d.v.r. and $X$ be an $A$-scheme as in the section
\ref{agreements}. The following result is due to Cesnavicius \cite[Proposition 4.1]{C}.
\begin{thm}(Proposition 4.1, [C])\label{geom_pres_mixed_char}
Let $x_1,x_2,\dots,x_n$ be closed points in the scheme $X$.
Let $Z$ be a closed subset in $X$ of codimension at least 2 in $X$.
Then there are an affine neighborhood $X^{\circ}$ of points
$x_1,x_2,\dots,x_n$, an open affine subscheme $S\subseteq \Aff^{d-1}_A$ and
{\bf a smooth $A$-morphism}
$$q: X^{\circ}\to S$$
of pure relative dimension $d-1$ such that
$Z^{\circ}$ is $S$-finite, where $Z^{\circ}=Z\cap X^{\circ}$.
\end{thm}

\section{Proof of the geometric case of Theorem \ref{thm2}}
Let $A$, $p>0$, $d\geq 1$, $X$, $x_1,x_2,\dots,x_n\in X$, $\mathcal O$ and $U$  be as in Section \ref{agreements}.
Write $\mathcal K$ for the fraction field of the ring $\mathcal O$.

\begin{thm}\label{thm2_geom}
Let $D$ be an Azumaya $\mathcal O$-algebra
and $Nrd_D: D^{\times}\to \mathcal O^{\times}$ be the reduced norm homomorphism.
Let $a \in \mathcal O^{\times}$. If $a$ is a reduced norm for the central simple
$\mathcal K$-algebra $D\otimes_{\mathcal O} \mathcal K$,
then $a$ is a reduced norm for the algebra $D$.
\end{thm}

\begin{prop}\label{Purity_for Azumaya}
Let $D$ be an Azumaya $\mathcal O$-algebra.
Let $K_*$ be the Quillen $K$-functor.
Then for each integer $n\geq 0$
the sequence is exact
\begin{equation}
K_n(D)\to K_n(D\otimes_{\mathcal O} \mathcal K)\xrightarrow{\partial} \oplus_{y\in U^{(1)}}K_{n-1}(D\otimes_{\mathcal O} k(y))
\end{equation}
Particularly, it is exact for $n=1$.
\end{prop}

\begin{proof}[Reducing Theorem \ref{thm2_geom} to Proposition \ref{Purity_for Azumaya}]
Consider the commutative diagram of groups
\begin{equation}
\label{RactangelDiagram}
    \xymatrix{
K_1(D) \ar[rr]^{\eta^*} \ar[d]_{Nrd}&& K_1(D\otimes_{\mathcal O} \mathcal K)\ar[rr]^{\partial} \ar[d]^{Nrd} &&
\oplus_{y\in U^{(1)}}K_0(D\otimes_{\mathcal O} k(y)) \ar[d]^{Nrd}  \\
\mathcal O^{\times} \ar[rr]^{\eta^*}&& \mathcal K^{\times} \ar[rr]^{\partial} && \oplus_{y\in U^{(1)}}K_0(k(y))\\ }
\end{equation}
By Proposition \ref{Purity_for Azumaya} the complex on the top is exact.
The bottom map $\eta^*$ is injective.
The right hand side vertical map $Nrd$ is injective. Thus, the map
$$\mathcal O^{\times}/Nrd(K_1(D))\to \mathcal K^{\times}/Nrd(K_1(D\otimes_{\mathcal O} \mathcal K))$$
is injective.
The image of the left vertical map coincides with
$Nrd(D^{\times})$
and the image of the middle vertical map coincides with
$Nrd((D\otimes_{\mathcal O} \mathcal K)^{\times})$. Thus, the map
$$\mathcal O^{\times}/Nrd(D^{\times})\to \mathcal K^{\times}/Nrd((D\otimes_{\mathcal O} \mathcal K)^{\times})$$
is injective. The derivation of Theorem \ref{thm2_geom} from Proposition
\ref{Purity_for Azumaya} is completed.
\end{proof}

\begin{proof}[Proof of Proposition \ref{Purity_for Azumaya}]
To prove this proposition
it sufficient to prove vanishing of the support extension map
$ext_{2,1}: K'_{n}(U;D)_{\geq 2}\to K'_{n}(U;D)_{\geq 1}$.
Prove that $ext_{2,1}=0$.
Take an $a\in K'_n(U;D)_{\geq 2}$.
We may assume that $a\in K'_n(Z;D)$ for a closed $Z$ in $U$ with $codim_U(Z)\geq 2$.
Enlarging $Z$ we may assume that each its irreducible component
contains at least one of the point $x_i$'s and still $codim_U(Z)\geq 2$.
Our aim is to find a closed subset $Z_{ext}$ in $U$ containing $Z$ such that
the element $a$ vanishes in $K'_n(Z_{ext};D)$ and
$codim_U(Z_{ext})\geq 1$. We will follow the method of \cite{PS}. \\
Let $\bar Z$ be the closure of $Z$ in $X$.
Shrinking $X$ and $\bar Z$ accordingly we may and will suppose that
$D$ is an Azumaya algebra over $X$ and there is an element $\tilde a \in K'_n(\bar Z;D)$
such that $\tilde a|_{Z}=a$.
By Theorem \ref{geom_pres_mixed_char} there are
an affine neighborhood $X^{\circ}$ of points
$x_1,x_2,\dots,x_n$, an open affine subscheme $S\subseteq \Aff^{d-1}_A$, a smooth $A$-morphism \\
$$q: X^{\circ}\to S$$
of relative dimension one
such that $Z^{\circ}/S$ is finite, where $Z^{\circ}=\bar Z\cap X^{\circ}$.
Put $a^{\circ}=\tilde a|_{Z^{\circ}}$.

Put $s_i=q(x_i)$. Consider the semi-local ring
$\mathcal O_{S,s_1,...,s_n}$, put $B=Spec \ \mathcal O_{S,s_1,...,s_n}$ and
$X_B=q^{-1}(B)\subset X^{\circ}$. Put $Z_B=Z^{\circ}\cap X_B$.
Write $q_B$ for $q|_{X_B}: X_B\to B$. Note that $Z_B$ is finite over $B$.
Since $B$ is semi-local, hence so is $Z_B$.

Let $W\subset X_B$ be an open containing $Z_B$.
Write $\mathcal Z_W$ for $Z_B\times_B W$.
Let
$\Pi: \mathcal Z_W\to W$
be the projection to $W$
and
$q_Z: \mathcal Z_W\to Z_B$ be the projection to $Z_B$.
Since $Z_B/B$ is finite the morphism $\Pi$ is finite.
Since $\Pi$ is finite the subset $Z_{new}:=\Pi(\mathcal Z_W)$ of $W$ is closed in $W$.
Since $W$ contains $Z_B$, we have an inclusion $i: Z_B\hookrightarrow Z_{new}$ of closed subsets in $W$.
Since $q_Z: \mathcal Z_W\to Z_B$ is smooth of relative dimension one $Z_{new}$ has codimension
at least one in $W$. Clearly, $Z_B$ contains all the points
$x_1,...,x_n$.
This shows that $W$ contains $U$ and $Z_B$ contains $Z$.
One can check that $Z=Z_B\cap U$. Write $j: Z\to Z_B$
for the inclusion.

Put $Z_{ext}=U\cap Z_{new}$. Since $Z$ is in $U\cap Z_{new}=Z_{ext}$, hence we have an inclusion
$in: Z\hookrightarrow Z_{ext}$ of closed subsets in $U$.
The inclusion $U\xrightarrow{can} W$ is a flat morphism. Hence the inclusions
$inj: Z_{ext}\to Z_{new}$ and $j: Z\to Z_B$ are also flat morphisms. Thus the homomorphisms
$inj^*: K'_n(Z_{new};D)\to K'_n(Z_{ext};D)$ and $j^*: K'_n(Z_B;D)\to K'_n(Z;D)$
are well-defined. Moreover $inj^*\circ i_*=in_*\circ j^*$.

As explained just above
the generic point of $W$ is not in $Z_{new}$. Thus, the generic point of $U$
is not in $Z_{ext}$ and $Z_{ext}$ has codimension
at least one in $W$.
Recall the Azumaya algebra $D$ is an Azumaya algebra over $X$. We still will write $D$ for $D|_W$.
Also we are given with an element $a_B:=a^{\circ}|_{Z_B}\in K'_n(Z_B;D)$ such that
$j^*(a_B)=a$ in $K'_n(Z;D)$.\\\\
{\it Claim}. There is an open $W$ in $X_B$ containing $Z_B$ such that for the closed inclusion
$i: Z_B\hookrightarrow Z_{new}$
the map $i_*: K'_n(Z_B;D)\to K'_n(Z_{new};D)$ vanishes. \\\\
Given this {\it Claim} complete the proof of the proposition as follows:
$in_*(a)=in_*(j^*(a_B))=inj^*(i_*(a_B))=inj^*(0)=0$ in $K'_n(Z_{ext};D)$ and $Z_{ext}\subsetneqq U$.\\\\
In the rest of the proof we prove the {\it Claim}. So, let $W\subset X_B$ be as in the {\it Claim}.
Let $\mathcal Z_W=Z_B\times_B W$ and the projections
$\Pi: \mathcal Z_W\to W$,
$q_Z: \mathcal Z_W\to Z_B$ be as above in this proof.
The closed embedding $in: Z_B\hookrightarrow W$ defines a section
$\Delta=(id\times in): Z_B\to \mathcal Z_W$
of the projection $q_Z$. Also one has an equality
$in=\Pi \circ \Delta$.
Put $D_Z=D|_{Z_B}$ and write ${_{Z}}D_W$ for the Azumaya algebra
$\Pi^*(D)\otimes q^*_Z(D^{op}_Z)$
over $\mathcal Z_W$.
The $\mathcal O_{\mathcal Z_W}$-module
$\Delta_*(D_Z)$
has an obvious left ${_{Z}}D_W$-module structure.
And it is equipped with an obvious epimorphism
$\pi: {_{Z}}D_X\to \Delta_*(D_Z)$
of the left ${_{Z}}D_W$-modules.
Following \cite{PS} one can see that $I:=Ker(\pi)$
is a left projective ${_{Z}}D_W$-module. Hence the left ${_{Z}}D_W$-module
$\Delta_*(D_Z)$ defines an element $[\Delta_*(D_Z)]=[{_{Z}}D_W]-[I]$
in $K_0(\mathcal Z_W;{_{Z}}D_W)$. This element has rank zero.
Hence by \cite{DeM} it vanishes semi-locally on $\mathcal Z_W$.
Thus, there is a neighborhood $\mathcal W$ of $Z_B\times_B Z_B$ in $\mathcal Z_W$
such that $\Delta_*(D_Z)$ vanishes in $K_0(\mathcal W;{_{Z}}D_W)$. It is easy to see that
$\mathcal W$ contains a neighborhood of $Z_B\times_B Z_B$ of the form $\mathcal Z_{W'}$,
where $W'\subset X_B$ is an open containing $Z_B$.\\\\
Thus, replacing notation we may suppose that
$W\subset X_B$ is as in the {\it Claim} and
the element $[\Delta_*(D_Z)]=[{_{Z}}D_W]-[I]$ vanishes
in $K_0(\mathcal Z_W;{_{Z}}D_W)$.
It remains to check that for this specific $W$ the map
$i_*: K'_n(Z_B;D)=K'_n(Z_B;D_Z)\to K'_n(Z_{new};D)$
vanishes.\\\\
The functor $(P,M)\mapsto P\otimes_{q^*_Z(D_Z)} M$
induces a bilinear pairing
$$\cup_{q^*_Z(D_Z)}: K_0(\mathcal Z_W;{_{Z}}D_W)\times K'_n(\mathcal Z_W;q^*_Z(D_Z))\to K'_n(\mathcal Z_W;\Pi^*(D))$$
Each element
$\alpha\in K_0(\mathcal Z_W;{_{Z}}D_W)$
defines a group homomorphism
$$\alpha_*=\Pi_*\circ ( \alpha \cup_{q^*_Z(D_Z)} - )\circ q^*_Z: K'_n(Z_B;D)=K'_n(Z_B;D_Z)\to K'_n(Z_{new};D)$$
which takes an element $b\in K'_n(Z_B;D_Z)$ to
the one
$\Pi_*(\alpha \cup_{q^*_Z(D_Z)} q^*_Z(b))$
in $K'_n(Z_{new};D)$.\\\\
Following \cite{PS} we see that the map $[\Delta_*(D_Z)]_*$ coincides with the map
$i_*: K'_n(Z_B;D)=K'_n(Z_B;D_Z)\to K'_n(Z_{new};D)$. The equality
$0=[\Delta_*(D_Z)]\in K_0(\mathcal Z_W;{_{Z}}D_W)$ proven just above
shows that the map $i_*$ vanishes. The {\it Claim} is proved. The proposition is proved.

\end{proof}

Let $A$, $p>0$, $d\geq 1$, $X$ be as in Section \ref{agreements}. Let $y_1,y_2,\dots,y_n\in X$ be points not necessary closed.
Write $\mathcal O_y$ for the semi-local ring $\mathcal O_{X,\{y_1,y_2,\dots,y_n\}}$ and $U_y$ for $Spec \mathcal \ O_y$.
Write $\mathcal K$ for the fraction field of the ring $\mathcal O_y$.
The following result can be derived from Theorem \ref{thm2_geom} in a standard way.
\begin{thm}\label{thm3_geom}
Let $D$ be an Azumaya $\mathcal O_y$-algebra
and $Nrd_D: D^{\times}\to \mathcal O^{\times}_y$ be the reduced norm homomorphism.
Let $a \in \mathcal O^{\times}_y$. If $a$ is a reduced norm for the central simple
$\mathcal K$-algebra $D\otimes_{\mathcal O_y} \mathcal K$,
then $a$ is a reduced norm for the algebra $D$.
\end{thm}

\begin{proof}[Proof of Theorem \ref{thm3_geom}]
One can choose closed points $x_1,x_2,...,x_n$ in $X$ such that $x_i$ is in the closure of $y_i$,
$D=\tilde D\otimes_{\mathcal O} \mathcal O_y$ for an Azumaya $\mathcal O$-algebra $\tilde D$,
where $\mathcal O:=\mathcal O_{X,\{x_1,x_2,\dots,x_n\}}$,
$a\in \mathcal O^{\times}$. In particular,  $\mathcal O\subseteq \mathcal O_y$.
Now by Theorem \ref{thm2_geom} the element $a$ is a reduced norm from $\tilde D$.
Thus, $a$ is a reduced norm from $\tilde D$.
\end{proof}

\begin{proof}[Proof of Theorem \ref{thm2}]
By Popescu
theorem \cite{Po}, \cite{Sw}, the ring R is a filtered direct limit of smooth $A$-algebras. Thus, a limit argument
allows us to assume that $R$ is the semilocalization of a smooth $A$-algebra at finitely many primes.
Theorem \ref{thm3_geom} completes the proof.

\end{proof}

\end{document}